\documentclass[12pt]{article}
\usepackage{amsmath,amssymb,amsthm}
\usepackage[margin=1in]{geometry} 
\usepackage{booktabs}
\usepackage{array}
\usepackage[hidelinks]{hyperref}
\theoremstyle{plain}
\newtheorem{theorem}{Theorem}[section]
\newtheorem{lemma}[theorem]{Lemma}
\newtheorem{proposition}[theorem]{Proposition}

\theoremstyle{definition}
\newtheorem{definition}{Definition}[section]

\title{A Computer-Assisted Proof that \(R(5,10;5,8)=21\) for the Generalized Threshold Ramsey Number}
\author{Wancai Guo, Min Wan\\[0.5ex]
	\small Department of Mathematics, College of Sciences,
	Shihezi University, Shihezi, Xinjiang 832003, China\\
	\small (\texttt{guowancai@xjshzu.com}, \texttt{wanmin@shzu.edu.cn})}

\date{}

\begin{document}
	\maketitle
	
	\begin{abstract}
		We prove that the generalized threshold Ramsey number \(R(5,10;5,8)\) equals \(21\). The proof is computer-assisted. We first enumerate all \((4,6;5,8)\)-good and \((5,10;4,4)\)-good graphs. Then we apply the lucky-vertex inequality derived from the McKay--Radziszowski edge identity to filter candidate pairs for a gluing search, and encode the remaining cases as SAT instances. The search yields exactly \(17\) non-isomorphic \((5,10;5,8;20)\)-good graphs. Finally, we exclude all one-vertex extensions of these \(17\) graphs, which proves that no \((5,10;5,8;21)\)-good graph exists.
	\end{abstract}
	
	\section{Problem Definition and Main Result}
	
	\subsection{Notation and Conventions}
	\label{sec:notation}
	
	All graphs in this paper are simple and undirected. For a graph \(G\), let \(\nu(G)\) denote the number of vertices and \(\varepsilon(G)\) the number of edges. The complement of \(G\) is denoted by \(\overline{G}\). For a vertex \(v\in V(G)\), let \(d_G(v)\) be its degree, \(N_G(v)\) its neighborhood, and \(N_{\overline{G}}(v)\) its antineighborhood. For a subset \(S\subseteq V(G)\), let \(G[S]\) be the subgraph induced by \(S\).
	
	\subsection{Definitions}
	
	\begin{definition}\label{def:R}
		Let \(s,m,t,n\) be positive integers. The generalized threshold Ramsey number \(R(s,m;t,n)\) is the smallest positive integer with the property that every graph of that order satisfies at least one of the following:
		
		(1) there exists \(S\subseteq V(G)\) with \(|S|=s\) and \(\varepsilon(G[S])\ge m\);
		
		(2) there exists \(T\subseteq V(G)\) with \(|T|=t\) and \(\varepsilon(\overline{G}[T])\ge n\).
	\end{definition}
	
	\begin{definition}[\((s,m;t,n)\)-good graph]\label{def:smtn}
		Let \(s,m,t,n\) be positive integers. A graph \(G\) is \((s,m;t,n)\)-good if
		
		(1) every \(s\)-vertex subset induces fewer than \(m\) edges in \(G\), i.e.,
		\[
		\varepsilon(G[S])\le m-1
		\]
		for every \(S\subseteq V(G)\) with \(|S|=s\);
		
		(2) every \(t\)-vertex subset induces fewer than \(n\) edges in the complement \(\overline{G}\), i.e.,
		\[
		\varepsilon(\overline{G}[T])\le n-1.
		\]
		If \(G\) has order \(\nu\), we call it an \((s,m;t,n;\nu)\)-good graph. Let
		\[
		\mathcal{R}(s,m;t,n;\nu)
		\]
		denote the set of all such graphs. For a nonnegative integer \(\varepsilon\), let
		\[
		\mathcal{R}(s,m;t,n;\nu,\varepsilon)
		\]
		denote the subset of \(\mathcal{R}(s,m;t,n;\nu)\) consisting of those graphs with exactly \(\varepsilon\) edges.
	\end{definition}

	\begin{theorem}\label{thm:main}
		\[
		R(5,10;5,8)=21.
		\]
	\end{theorem}
	
	Theorem \ref{thm:main} follows from the following two statements:
	
	(1) there exists a \((5,10;5,8)\)-good graph of order \(20\), hence \(R(5,10;5,8)\ge 21\);
	
	(2) there is no \((5,10;5,8)\)-good graph of order \(21\), hence \(R(5,10;5,8)\le 21\).
	
	\medskip
	\noindent
	The number \(R(5,10;5,8)\) is closely related to the classical Ramsey numbers
	\(R(K_5, K_5-P_3)=25\) and \(R(K_5, W_5)=27\); see \cite{radziszowski}.
	Indeed, \(K_5-P_3\) and \(W_5\) are exactly the two \(5\)-vertex graphs with \(8\) edges,
	and a \((5,10;5,8)\)-good graph is precisely a \(K_5\)-free graph whose complement
	contains neither \(K_5-P_3\) nor \(W_5\). Therefore
	\[
	R(5,10;5,8) = R(K_5, \{K_5-P_3, W_5\}).
	\]
	This is the set Ramsey number \(r_5(10,8)\) for five-vertex graphs with fixed numbers of edges; see Harborth \cite{Harborth1993}.
	To our knowledge, this set Ramsey number has not been determined before.
	Our result \(R(5,10;5,8)=21\) thus gives its exact value, which is strictly smaller
	than both \(R(K_5, K_5-P_3)=25\) and \(R(K_5, W_5)=27\).
	
	\section{Main Ideas}
	The decomposition into the neighborhood and antineighborhood of a vertex, together with the subsequent gluing and one-vertex extension, are inspired by the methods of McKay and Radziszowski \cite{MR1995}.
	
	\subsection{Induced subgraphs at a vertex}
	
	If \(G\in \mathcal{R}(5,10;5,8;\nu)\) and \(v\in V(G)\) has degree \(d\), let
	\[
	G_v^+ = G[N_G(v)]
	\]
	be the subgraph induced by the neighborhood of \(v\), and let
	\[
	G_v^- = G[N_{\overline{G}}(v)]
	\]
	be the subgraph induced by the antineighborhood of \(v\).
	
	\begin{lemma}\label{lem:induced}
		Let \(s,m,t,n,\nu\) be positive integers with \(m\ge s\) and \(n\ge t\).
		Let \(G \in \mathcal{R}(s,m;t,n;\nu)\), and let \(v \in V(G)\) have degree \(d\). Then
		\[
		G_v^+ \in \mathcal{R}(s-1,\,m-s+1;\,t,\,n;\,d)
		\]
		and
		\[
		G_v^- \in \mathcal{R}(s,\,m;\,t-1,\,n-t+1;\,\nu-1-d).
		\]
	\end{lemma}
	
	\begin{proof}
		Since \(G\) is \((s,m;t,n)\)-good, every \(s\)-vertex subset of \(G\) induces at most \(m-1\) edges, and every \(t\)-vertex subset of \(\overline{G}\) induces at most \(n-1\) edges.
		
		First consider \(G_v^+\). Let \(S\subseteq V(G_v^+)\) with \(|S|=s-1\). Then \(S\cup\{v\}\) is an \(s\)-vertex subset of \(G\), so
		\[
		\varepsilon(G[S\cup\{v\}])\le m-1.
		\]
		Since \(v\) is adjacent in \(G\) to every vertex of \(S\), we have
		\[
		\varepsilon(G[S])+(s-1)\le m-1,
		\]
		hence \(\varepsilon(G[S])\le m-s\). Also, for every \(T\subseteq V(G_v^+)\) with \(|T|=t\),
		\[
		\varepsilon(\overline{G_v^+}[T])=\varepsilon(\overline{G}[T])\le n-1.
		\]
		Thus \(G_v^+\in\mathcal{R}(s-1,m-s+1;t,n;d)\).
		
		The argument for \(G_v^-\) is analogous, using the complement. Let \(T\subseteq V(G_v^-)\) with \(|T|=t-1\). Since \(v\) is non-adjacent in \(G\) to every vertex of \(T\), the set \(T\cup\{v\}\) is a \(t\)-vertex subset of \(G\) with
		\[
		\varepsilon(G[T\cup\{v\}])=\varepsilon(G[T]).
		\]
		Because \(G\) is \((s,m;t,n)\)-good, we have
		\[
		\varepsilon(\overline{G}[T\cup\{v\}])\le n-1.
		\]
		In \(\overline{G}\), the vertex \(v\) is adjacent to every vertex of \(T\), so
		\[
		\varepsilon(\overline{G}[T])+(t-1)\le n-1,
		\]
		hence \(\varepsilon(\overline{G}[T])\le n-t\). Also, every \(s\)-vertex subset of \(G_v^-\) is an \(s\)-vertex subset of \(G\), so it induces at most \(m-1\) edges. Therefore \(G_v^-\in\mathcal{R}(s,m;t-1,n-t+1;\nu-1-d)\).
	\end{proof}
	
	For our main problem we use the special case \(s=5\), \(m=10\), \(t=5\), \(n=8\), which gives
	\[
	G_v^+ \in \mathcal{R}(4,6;5,8;d)
	\quad\text{and}\quad
	G_v^- \in \mathcal{R}(5,10;4,4;\nu-1-d).
	\]	
	
	One idea is to find all \((4,6;5,8)\)-good graphs and all \((5,10;4,4)\)-good graphs, and then make a census of both classes.
	
	For computation, we save all graphs in graph6 format. For the census of \((4,6;5,8)\)-good graphs, we use the method of one-vertex extension. Indeed, if a graph \(G\) of order \(\nu+1\) is \((4,6;5,8;\nu+1)\)-good, then deleting any vertex \(v\) produces a graph \(G'\) that is \((4,6;5,8;\nu)\)-good. Starting from \(K_1\), we use one-vertex extension to obtain all \((4,6;5,8)\)-good graphs. 
	
	The census of \((5,10;4,4)\)-good graphs is carried out in the same way. These two censuses serve complementary roles in the decomposition of Lemma~\ref{lem:induced}: the \((4,6;5,8)\)-good graphs describe the possible neighborhoods \(G_v^+\), while the \((5,10;4,4)\)-good graphs describe the possible antineighborhoods \(G_v^-\). Their total counts are collected in Table~\ref{tab:census-total}.
	
	\begin{table}[h]
		\centering
		\begin{tabular}{|c||c|c||c|c|}
			\hline
			\(\nu\) & \(|\mathcal R(4,6;5,8;\nu)|\) & \(\varepsilon\) & \(|\mathcal R(5,10;4,4;\nu)|\) & \(\varepsilon\) \\
			\hline
			1  & 1     & [0,0]  & 1  & [0,0] \\
			2  & 2     & [0,1]  & 2  & [0,1] \\
			3  & 4     & [0,3]  & 4  & [0,3] \\
			4  & 10    & [0,5]  & 7  & [3,6] \\
			5  & 25    & [3,8]  & 12 & [5,9] \\
			6  & 83    & [5,12] & 21 & [9,13] \\
			7  & 311   & [7,16] & 31 & [13,18] \\
			8  & 1445  & [11,21]& 43 & [18,24] \\
			9  & 6867  & [15,27]& 40 & [24,30] \\
			10 & 24785 & [19,33]& 18 & [30,37] \\
			11 & 35191 & [25,40]& 2  & [44,45] \\
			12 & 6336  & [30,48]& 1  & [54,54] \\
			13 & 20    & [39,45]& 0  & -- \\
			\hline
		\end{tabular}
		\caption{A census of \(\mathcal R(4,6;5,8)\) and \(\mathcal R(5,10;4,4)\).}
		\label{tab:census-total}
	\end{table}
	
	\subsection{Gluing and the lucky-vertex inequality}
	
	From the computation, we find that the maximum order of a \((4,6;5,8)\)-good graph is \(13\), and the maximum order of a \((5,10;4,4)\)-good graph is \(12\).
	
	If \(G\in\mathcal{R}(5,10;5,8;20)\), then for every \(v\in V(G)\), we have
	\[
	\nu(G_v^+)\le 13,\qquad \nu(G_v^-)\le 12.
	\]
	Also,
	\[
	d_G(v)+d_{\overline{G}}(v)=19.
	\]
	Hence
	\[
	7\le d_G(v)\le 13,
	\]
	i.e., letting \(d=d_G(v)\), we have
	\[
	d\in\{7,8,9,10,11,12,13\}.
	\]
	For \(d\in\{7,8,9,13\}\), we glue
	\[
	G_v^+\in\mathcal{R}(4,6;5,8;d)
	\]
	and
	\[
	G_v^-\in\mathcal{R}(5,10;4,4;19-d).
	\]
	At this point, the only undetermined edges of \(G\) are those whose endpoints lie in \(G_v^+\) and \(G_v^-\), respectively. We introduce a Boolean variable for each such edge, where the value \(1\) indicates that the edge is present in \(G\) and \(0\) indicates that it is absent. We then encode the resulting constraints as a SAT instance and use a SAT solver to test whether a \((5,10;5,8;20)\)-good graph can be constructed. It is easy to see that we only need to consider those \(5\)-vertex subsets that may contain undetermined edges, and check whether a solution can make all such \(5\)-vertex subsets satisfy
	\[
	3\le \varepsilon \le 9.
	\]
	
	For \(d\in\{10,11,12\}\), we use an inequality as a necessary condition to pre-filter the candidate pairs before the SAT search.
	
	In \cite{MR1997}, McKay and Radziszowski gave the following identity: for any graph \(G\) of order \(\nu\),
	\begin{equation}\label{eq:MR-identity}
		\sum_{v\in V(G)} 2\varepsilon(G_v^-)
		=
		\sum_{v\in V(G)} \left(\nu(G_v^+)\bigl(\nu-2\nu(G_v^+)\bigr)
		+2\varepsilon(G_v^+)\right).
	\end{equation}
	Since the two sums in \eqref{eq:MR-identity} are equal, there exists at least one vertex \(v\in V(G)\) satisfying
	\begin{equation}\label{eq:lucky-ineq}
		2\varepsilon(G_v^-)\le \nu(G_v^+)\bigl(\nu-2\nu(G_v^+)\bigr)+2\varepsilon(G_v^+).
	\end{equation}
	We call such a vertex a \emph{lucky vertex}.
	
	As will be shown in Section~3.1, the gluing search for \(d\in\{7,8,9,13\}\), which does not use the lucky-vertex inequality \eqref{eq:lucky-ineq}, yields no solution. Hence if a \((5,10;5,8;20)\)-good graph \(G\) exists, then every vertex \(x\in V(G)\) has degree
	\[
	d(x)\in\{10,11,12\}.
	\]
	Since every counterexample must contain a lucky vertex, there is at least one vertex \(v\in V(G)\) with \(d\in\{10,11,12\}\) for which \eqref{eq:lucky-ineq} holds.
	
	Fix \(d\in\{10,11,12\}\), and consider a candidate pair
	\[
	A=G_v^+\in\mathcal{R}(4,6;5,8;d),\qquad
	B=G_v^-\in\mathcal{R}(5,10;4,4;19-d)
	\]
	for a lucky vertex \(v\). Writing \(\varepsilon(A)=\varepsilon(G_v^+)\) and \(\varepsilon(B)=\varepsilon(G_v^-)\), and using \(\nu=20\), the inequality \eqref{eq:lucky-ineq} becomes the explicit bound
	\[
	2\varepsilon(B)\le d\,(20-2d)+2\varepsilon(A),
	\]
	that is,
	\[
	\varepsilon(B)\le \varepsilon(A)+\tfrac{1}{2}d\,(20-2d).
	\]
	For \(d=10,11,12\), the constant \(\tfrac{1}{2}d(20-2d)\) equals \(0,-11,-24\), respectively. Any pair \((A,B)\) violating the corresponding bound cannot arise from a lucky vertex and is discarded before the SAT search. In this way the pre-filtering by the lucky-vertex inequality \eqref{eq:lucky-ineq} eliminates the vast majority of the candidate pairs; for instance, when \(d=10\) it retains only \(16.18\%\) of the \(24785\cdot 40=991400\) pairs.
	
	For the surviving pairs, the only undetermined edges of \(G\) are those with one endpoint in \(A\) and the other in \(B\). Encode these edges as Boolean variables and use a SAT solver to decide whether a \((5,10;5,8;20)\)-good graph can be constructed. It is easy to see that we only need to consider those \(5\)-vertex subsets that may contain undetermined edges, and check whether a solution can make all such \(5\)-vertex subsets satisfy
	\[
	3\le \varepsilon \le 9.
	\]
	In the SAT instance we additionally impose the degree constraints
	\[
	d(x)\in\{10,11,12\}\qquad\text{for every }x\in V(G)\setminus\{v\},
	\]
	which follow from the fact that \(d(x)\in\{10,11,12\}\) for every vertex of any counterexample. The search is implemented in PySAT \cite{pysat} using the CaDiCaL backend \cite{cadical}.
	
	\section{Computational Results}	
	
	\subsection{The cases \(d\in\{7,8,9,13\}\)}
	
	For these four values of \(d\), we ran the gluing search without the lucky-vertex inequality \eqref{eq:lucky-ineq}. All candidate pairs \((A,B)\) with
	\[
	A\in\mathcal{R}(4,6;5,8;d),\qquad
	B\in\mathcal{R}(5,10;4,4;19-d)
	\]
	were checked with the SAT solver. The numbers are collected in Table~\ref{tab:no-lucky}.
	
	\begin{table}[h]
		\centering
		\begin{tabular}{@{}ccc@{}}
			\toprule
			\(d\) & candidate pairs & outcome \\
			\midrule
			7  & 311          & all UNSAT \\
			8  & 2890         & all UNSAT \\
			9  & 123606       & all UNSAT \\
			13 & 420          & all UNSAT \\
			\bottomrule
		\end{tabular}
		\caption{The cases \(d\in\{7,8,9,13\}\): no \((5,10;5,8;20)\)-good graph.}
		\label{tab:no-lucky}
	\end{table}
	
	Consequently, no \((5,10;5,8;20)\)-good graph contains a vertex of degree \(d\in\{7,8,9,13\}\). Combined with the degree restriction already obtained in Section~2, this gives the refined degree restriction
	\[
	d(x)\in\{10,11,12\}\qquad\text{for every }x\in V(G).
	\]
	This is the restriction used in the next three subsections.
	
	\subsection{The case \(d=10\)}
	
	We now use the lucky-vertex inequality \eqref{eq:lucky-ineq} from Section~2. For \(d=10\), the constant \(\tfrac12 d(20-2d)\) vanishes, and the inequality \eqref{eq:lucky-ineq} reduces to
	\[
	\varepsilon(B)\le \varepsilon(A).
	\]
	Of the \(24785\cdot 40=991400\) pairs, \(160369\) survive this filter (\(16.18\%\)). For each surviving pair, we enforce
	\[
	3\le \varepsilon(G[S])\le 9
	\]
	for every \(5\)-vertex subset \(S\) that contains at least one vertex of \(A\) and at least one vertex of \(B\), together with the degree constraints \(d(x)\in\{10,11,12\}\) for \(x\neq v\).
	
	All \(160369\) instances returned UNSAT. Hence no \((5,10;5,8;20)\)-good graph has a lucky vertex of degree \(10\).
	
	\subsection{The case \(d=11\)}
	
	For \(d=11\), the constant \(\tfrac12 d(20-2d)\) equals \(-11\), and the inequality \eqref{eq:lucky-ineq} becomes
	\[
	\varepsilon(B)\le \varepsilon(A)-11.
	\]
	Of the \(35191\cdot 43=1513213\) pairs, \(395615\) survive (\(26.14\%\)). The surviving pairs were passed to the SAT solver. Out of \(395615\) instances, \(9\) returned SAT. Every SAT solution extends to a \((5,10;5,8;20)\)-good graph. Blocking clauses were then used to enumerate all solutions within each of the \(9\) SAT pairs; after canonicalization we obtained exactly \(7\) non-isomorphic \((5,10;5,8;20)\)-good graphs whose lucky vertex has degree \(11\).
	
	\begin{table}[htbp]
		\centering
		\begin{tabular}{|c|c||c|c|}
			\hline
			\(\varepsilon\) & \(|\mathcal R(4,6;5,8;11,\varepsilon)|\) & \(\varepsilon\) & \(|\mathcal R(5,10;4,4;8,\varepsilon)|\) \\
			\hline
			25 & 7    & 18 & 1  \\
			26 & 106  & 19 & 6  \\
			27 & 927  & 20 & 13 \\
			28 & 4006 & 21 & 13 \\
			29 & 8531 & 22 & 7  \\
			30 & 9846 & 23 & 2  \\
			31 & 6641 & 24 & 1  \\
			32 & 3092 &    &    \\
			33 & 1234 &    &    \\
			34 & 497  &    &    \\
			35 & 194  &    &    \\
			36 & 74   &    &    \\
			37 & 25   &    &    \\
			38 & 8    &    &    \\
			39 & 2    &    &    \\
			40 & 1    &    &    \\
			\hline
		\end{tabular}
		\caption{Edge-count distributions for \(d=11\): \(\mathcal R(4,6;5,8;11,\varepsilon)\) and \(\mathcal R(5,10;4,4;8,\varepsilon)\).}
		\label{tab:d11-edges}
	\end{table}
	
	\subsection{The case \(d=12\)}
	
	For \(d=12\), the constant \(\tfrac12 d(20-2d)\) equals \(-24\), and the inequality \eqref{eq:lucky-ineq} becomes
	\[
	\varepsilon(B)\le \varepsilon(A)-24.
	\]
	Of the \(6336\cdot 31=196416\) pairs, \(20192\) survive (\(10.28\%\)). Running SAT on the \(20192\) surviving pairs yielded \(13\) SAT instances. Enumerating all solutions within these \(13\) pairs and canonicalizing the resulting graphs gives exactly \(10\) non-isomorphic \((5,10;5,8;20)\)-good graphs whose lucky vertex has degree \(12\). Every one of these \(10\) graphs is \(12\)-regular, i.e.\ every vertex has degree \(12\).
	
	\subsection{The complete list for order \(20\)}
	
	Combining the two lists above, we obtain the following census.
	
	\begin{table}[h]
		\centering
		\begin{tabular}{@{}ccc@{}}
			\toprule
			source of the graph & number of isomorphism classes & degree sequence \\
			\midrule
			\(d=11\)    & \(7\)  & contains a vertex of degree \(11\) \\
			\(d=12\)    & \(10\) & \(12\)-regular \\
			\midrule
			total     & \(17\) & \\
			\bottomrule
		\end{tabular}
		\caption{All \((5,10;5,8;20)\)-good graphs up to isomorphism.}
		\label{tab:census-20}
	\end{table}
	
	No graph from the \(d=11\) list can be isomorphic to a graph from the \(d=12\) list, because the degree sequences differ. Hence the \(17\) isomorphism classes listed above are pairwise distinct.
	
	It also follows that
	\[
	R(5,10;5,8)\ge 21.
	\]
	
	\section{Excluding Order \(21\)}
	
	Suppose, for contradiction, that \(F\) is a \((5,10;5,8;21)\)-good graph. For any vertex \(v\in V(F)\), the induced subgraph \(F-v\) is a \((5,10;5,8;20)\)-good graph. Indeed, every \(5\)-vertex subset of \(F-v\) is also a \(5\)-vertex subset of \(F\), 
	so it induces at most \(9\) edges in \(F-v\) and at most \(7\) edges in the complement. By the census of Section~3, \(F-v\) must therefore be isomorphic to one of the \(17\) graphs in Table~\ref{tab:census-20}.
	
	Thus it suffices to check, for each of the \(17\) good graphs \(G\) on \(20\) vertices, whether there is a graph \(F\) on \(21\) vertices obtained by adding one vertex to \(G\) such that \(F\) is again \((5,10;5,8)\)-good. Write \(u\) for the new vertex. There are \(20\) Boolean variables \(x_w\) (\(w\in V(G)\)), where \(x_w=1\) means that \(uw\in E(F)\). The conditions on \(F\) reduce to:

	(1) For every \(4\)-vertex subset \(S\subseteq V(G)\),
	\[
	3\le \varepsilon(G[S])+\sum_{w\in S}x_w\le 9.
	\]
	(This is the condition that the \(5\)-vertex subset \(S\cup\{u\}\) of \(F\) has between \(3\) and \(9\) edges.)
	
	(2) Every \(5\)-vertex subset of \(V(G)\) has between \(3\) and \(9\) edges in \(G\); this holds by assumption on \(G\).

	No further constraint is needed, since any \(5\)-vertex subset of \(F\) either lies entirely in \(V(G)\) or consists of \(u\) together with a \(4\)-vertex subset of \(V(G)\).
	
	For each of the \(17\) graphs, we built the corresponding SAT instance and called CaDiCaL through PySAT. All \(17\) instances returned UNSAT. This establishes the following.
	
	\begin{proposition}\label{prop:no-21}
		There is no \((5,10;5,8)\)-good graph of order \(21\).
	\end{proposition}
	
	\begin{proof}
		If such a graph \(F\) existed, then \(F-v\) for any \(v\in V(F)\) would be one of the \(17\) graphs in Table~\ref{tab:census-20}, and \(F\) would be a valid one-vertex extension of that graph. The SAT search above shows that no such extension exists. Contradiction.
	\end{proof}
	
	We then obtain the main result:
	\[
	R(5,10;5,8)=21.
	\]
	
	\section{A Census of Small Set Ramsey Numbers \(r_5(m,n)\)}
	
	The set Ramsey number \(r_5(m,n)\) is the smallest integer \(r\) such that for
	every graph \(G\) of order \(r\), either \(G\) contains a subgraph on \(5\)
	vertices with \(m\) edges, or \(\overline{G}\) contains a subgraph on \(5\)
	vertices with \(n\) edges \cite{Harborth1993}.
	Equivalently, in the notation of this paper,
	\[
	r_5(m,n) = R(5,m;5,n).
	\]
	In Table~\ref{tab:r5mn} we collect the current knowledge of \(r_5(m,n)\) for
	\(1\le m,n\le 10\). The baseline entries originate from the tabulation by 
	Harborth \cite{Harborth1993}. The entry \(r_5(10,8)=21\) is the main result of this paper,
	i.e., \(R(5,10;5,8)=21\). Using the same method, we also obtain the values 
	\(r_5(10,7)=17\) and \(r_5(9,8)=17\). For the specific cases \(m,n \in \{9,10\}\), 
	we incorporate the updates tracked in the dynamic survey of Radziszowski 
	\cite{radziszowski}, where the previously listed ranges \(r_5(10,9)=30\)--\(34\) 
	and \(r_5(10,10)=43\)--\(52\) have been refined to \(r_5(10,9)=30\) and 
	\(r_5(10,10)=43\)--\(46\). By symmetry, \(r_5(m,n)=r_5(n,m)\), so only the 
	upper triangle is shown.
	
	\begin{table}[htbp]
		\centering
		\small
		\setlength{\arrayrulewidth}{0.4pt}
		\renewcommand{\arraystretch}{1.4}
		\begin{tabular}{|c||*{10}{w{c}{1.45em}|}}
			\hline
			\(n \backslash m\) & 1 & 2 & 3 & 4 & 5 & 6 & 7 & 8 & 9 & 10 \\
			\hline\hline
			1  & 5 & 5 & 5 & 5 & 5 & 5 & 5 & 5 & 5 & 5 \\
			\hline
			2  &   & 5 & 5 & 5 & 5 & 5 & 5 & 5 & 5 & 6 \\
			\hline
			3  &   &   & 5 & 5 & 5 & 5 & 5 & 5 & 7 & 9 \\
			\hline
			4  &   &   &   & 5 & 5 & 5 & 5 & 6 & 7 & 9 \\
			\hline
			5  &   &   &   &   & 5 & 5 & 7 & 8 & 10 & 13 \\
			\hline
			6  &   &   &   &   &   & 6 & 7 & 9 & 11 & 14 \\
			\hline
			7  &   &   &   &   &   &   & 10 & 11 & 13 & 17 \\
			\hline
			8  &   &   &   &   &   &   &    & 14 & 17 & 21 \\
			\hline
			9  &   &   &   &   &   &   &    &    & 22 & 30 \\
			\hline
			10 &   &   &   &   &   &   &    &    &    & 43--46 \\
			\hline
		\end{tabular}
		\caption{The set Ramsey numbers \(r_5(m,n)\).}
		\label{tab:r5mn}
	\end{table}
	
	\section{Concluding Remarks}
	
	The proof of \(R(5,10;5,8)=21\) uses three ingredients:
	\begin{enumerate}
		\item a census of \((4,6;5,8)\)-good and \((5,10;4,4)\)-good graphs, obtained by one-vertex extension starting from \(K_1\);
		\item a gluing SAT search for \((5,10;5,8;20)\)-good graphs, using the lucky-vertex inequality as a pre-filter and degree constraints as an additional check;
		\item a one-vertex extension search from each of the \(17\) graphs in the order-\(20\) census, which excludes order~\(21\).
	\end{enumerate}
	
	The total wall-clock time of the computation was under two days on a standard laptop. All searches were implemented in SageMath \cite{sagemath}, using PySAT \cite{pysat} for the SAT instances and nauty \cite{nauty} for canonical labelling.
	
	It seems plausible that the same three-step approach extends to other small generalized threshold Ramsey numbers whose two auxiliary censuses terminate at a modest order. Larger parameters would require additional engineering, such as interval-based collapsing rules in the style of McKay--Radziszowski, to remain feasible.

	\section*{Acknowledgements}
	
	The authors thank Chuang Si from the School of Mathematical Sciences, Nanjing Normal University, for reproducing the computations, including the 17 counterexamples of order 20 and an independent verification of the upper bound by directly gluing to 21 vertices and finding no solution. The authors also thank DeepSeek for assisting in writing portions of the Python code used in the computations. The authors have independently verified all code and take full responsibility for its correctness.


\begin{thebibliography}{9}
		\bibitem{MR1995} B.~D. McKay and S.~P. Radziszowski.
		\(R(4,5)=25\).
		\emph{J. Graph Theory}, 19(3):309--322, 1995.
		
		\bibitem{Harborth1993} H.~Harborth.
		Ramsey numbers for sets of five vertex graphs with fixed numbers of edges.
		In N.~W. Sauer et al., editors, \emph{Finite and Infinite Combinatorics
			in Sets and Logic}, NATO ASI Series C, vol.~411, pp.~169--174.
		Kluwer Academic Publishers, Dordrecht, 1993.
		
		\bibitem{MR1997} B.~D. McKay and S.~P. Radziszowski.
		Subgraph counting identities and Ramsey numbers.
		\emph{J. Combin. Theory Ser. B}, 69(2):193--209, 1997.
		
		\bibitem{radziszowski} S.~P. Radziszowski. Small Ramsey numbers.
		\emph{Electron. J. Combin.}, 1:Dynamic Survey 1, 30, 1994
		(continuously updated).
		
		\bibitem{cadical} A.~Biere, K.~Fazekas, M.~Fleury, and M.~Heisinger.
		CaDiCaL, Kissat, Paracooba, Plingeling and Treengeling entering the SAT Competition 2020.
		In \emph{Proc. of SAT Competition 2020 -- Solver and Benchmark Descriptions},
		volume B-2020-1 of Department of Computer Science Report Series B,
		pp.~51--53, University of Helsinki, 2020.
		
		\bibitem{pysat} A.~Ignatiev, A.~Morgado, and J.~Marques-Silva.
		PySAT: A Python toolkit for prototyping with SAT oracles.
		In \emph{Theory and Applications of Satisfiability Testing -- SAT 2018},
		LNCS 10957, pp.~428--437, Springer, 2018.
		
		\bibitem{nauty} B.~D. McKay and A.~Piperno.
		Practical graph isomorphism, II.
		\emph{J. Symbolic Comput.}, 60:94--112, 2014.
		
		\bibitem{sagemath} W.~A. Stein et al.
		\emph{Sage Mathematics Software}.
		The Sage Development Team.
		\url{https://www.sagemath.org}.
	\end{thebibliography}
\end{document}